\documentclass[11pt]{amsart}
\usepackage{amsfonts,amssymb,amscd,amsmath,enumerate,verbatim,calc,latexsym,pstcol,pst-plot,pst-3d,}

\def\frk{\frak}               

\def\mm{{\frk m}}

\def\Phi{{\frk n}}
\def\Phi{{\frk N}}
\def\Fc{{\mathcal F}}
\def\opn#1#2{\def#1{\operatorname{#2}}} 
\opn\chara{char} \opn\length{\ell} \opn\pd{pd} \opn\rk{rk}
\opn\projdim{proj\,dim} \opn\injdim{inj\,dim} \opn\rank{rank}
\opn\depth{depth} \opn\grade{grade} \opn\height{height}
\opn\embdim{emb\,dim} \opn\codim{codim}

\opn\Tr{Tr} \opn\bigrank{big\,rank}
\opn\superheight{superheight}\opn\lcm{lcm}
\opn\trdeg{tr\,deg}
\opn\reg{reg} \opn\lreg{lreg} \opn\ini{in} \opn\lpd{lpd}
\opn\size{size} \opn\sdepth{sdepth}
\opn\link{link}\opn\fdepth{fdepth}\opn\lex{lex}
\opn\div{div} \opn\Div{Div} \opn\cl{cl} \opn\Cl{Cl}
\opn\Spec{Spec} \opn\Supp{Supp} \opn\supp{supp} \opn\Sing{Sing}
\opn\Ass{Ass} \opn\Min{Min}\opn\Mon{Mon}
\opn\Ann{Ann} \opn\Rad{Rad} \opn\Soc{Soc}
\opn\Im{Im} \opn\Ker{Ker} \opn\Coker{Coker} \opn\Am{Am}
\opn\Hom{Hom} \opn\Tor{Tor} \opn\Ext{Ext} \opn\End{End}
\opn\Aut{Aut} \opn\id{id}

\opn\nat{nat}
\opn\pff{pf}
\opn\Pf{Pf} \opn\GL{GL} \opn\SL{SL} \opn\mod{mod} \opn\ord{ord}
\opn\Gin{Gin} \opn\Hilb{Hilb}\opn\sort{sort}
\opn\aff{aff} \opn\con{conv} \opn\relint{relint} \opn\st{st}
\opn\lk{lk} \opn\cn{cn} \opn\core{core} \opn\vol{vol}
\opn\link{link} \opn\star{star}\opn\lex{lex}\opn\set{set}
\opn\gr{gr}

\opn\dstab{dstab}
\def\pot#1#2{#1[\kern-0.28ex[#2]\kern-0.28ex]}

\opn\dirlim{\underrightarrow{\lim}}
\opn\inivlim{\underleftarrow{\lim}}
\let\Sect=\bigcap

\def\Implies{\ifmmode\Longrightarrow \else
        \unskip${}\Longrightarrow{}$\ignorespaces\fi}
\def\implies{\ifmmode\Rightarrow \else
        \unskip${}\Rightarrow{}$\ignorespaces\fi}
\def\iff{\ifmmode\Longleftrightarrow \else
        \unskip${}\Longleftrightarrow{}$\ignorespaces\fi}

\let\:=\colon
\newtheorem{Theorem}{Theorem}[section]

\newtheorem{Corollary}[Theorem]{Corollary}
\newtheorem{Proposition}[Theorem]{Proposition}

\newtheorem{Examples}[Theorem]{Examples}

\let\epsilon\varepsilon
\let\kappa=\varkappa
\def\qed{\ifhmode\textqed\fi
      \ifmmode\ifinner\quad\qedsymbol\else\dispqed\fi\fi}
\def\textqed{\unskip\nobreak\penalty50
       \hskip2em\hbox{}\nobreak\hfil\qedsymbol
       \parfillskip=0pt \finalhyphendemerits=0}
\def\dispqed{\rlap{\qquad\qedsymbol}}

\opn\dis{dis}
\def\pnt{{\raise0.5mm\hbox{\large\bf.}}}

\opn\Lex{Lex}

\numberwithin{equation}{section}

\begin{document}

\title {Bounding the socles of powers of squarefree monomial ideals}

\author {J\"urgen Herzog and  Takayuki Hibi}

\address{J\"urgen Herzog, Fachbereich Mathematik, Universit\"at Duisburg-Essen, Campus Essen, 45117
Essen, Germany} \email{juergen.herzog@uni-essen.de}

\address{Takayuki Hibi, Department of Pure and Applied Mathematics, Graduate School of Information Science and Technology,
Osaka University, Toyonaka, Osaka 560-0043, Japan}
\email{hibi@math.sci.osaka-u.ac.jp}


\begin{abstract}
Let $S=K[x_1,\ldots,x_n]$ be the polynomial ring in $n$ variables 
over a field $K$ and $I\subset S$ a squarefree monomial ideal.  
In the present paper we are interested in the monomials $u \in S$ 
belonging to the socle $\Soc(S/I^{k})$ of $S/I^{k}$, i.e., $u \not\in I^{k}$ and
$ux_{i} \in I^{k}$ for $1 \leq i \leq n$.  We prove that if
a monomial $x_1^{a_1}\cdots x_n^{a_n}$ belongs to $\Soc(S/I^{k})$, then
$a_i\leq k-1$ for all $1 \leq i \leq n$.
We then discuss squarefree monomial ideals $I \subset S$ for which
$x_{[n]}^{k-1} \in \Soc(S/I^{k})$, where
$x_{[n]} = x_{1}x_{2}\cdots x_{n}$.
Furthermore, we give a combinatorial 
characterization of finite graphs $G$ on $[n] = \{1, \ldots, n\}$
for which $\depth S/(I_{G})^{2}=0$, where $I_{G}$
is the edge ideal of $G$.
\end{abstract}
\maketitle
\section*{Introduction}
The depth of powers of an ideal (especially, a monomial ideal) of the polynomial ring 
has been studied by many authors. 
In the present paper, we are interested in the socle of powers of a squarefree monomial ideal.  

Let $K$ be a field, $S=K[x_1,\ldots,x_n]$ the polynomial ring in $n$ variables 
over $K$,
and $I\subset S$ a graded ideal. We denote by $\mm=(x_1,\ldots,x_n)$
the graded maximal ideal of $S$. An element $f+I\in  S/I$  
is called a {\em socle element} of $S/I$ if $x_if\in I$ for $i=1,\ldots,n$. 
Thus $f+I$
is a non-zero socle element of $S/I$ if $f\in I:\mm\setminus I$. 
The set of socle elements $\Soc(S/I)$ of $S/I$ is called the {\em socle} of $S/I$. 
Notice that $\Soc(S/I)$ is a $K$-vector space isomorphic to $(I:\mm)/I$.  
One has $\depth S/I = 0$ if and only if $\Soc(S/I) \neq \{ 0 \}$.

In the case that $I$ is a monomial ideal, a case which we mainly consider here, 
$\Soc(S/I)$ is generated by the residue classes of monomials.  
If  $u$ and $v$ are monomials not belonging to $I$,  
then $u+I=v+I$, if and only if $u=v$. Thus, if $u$ is a monomial,
it is convenient to write $u\in \Soc(S/I)$
and to call $u$ a socle element of $S/I$ if $u+I\in \Soc(S/I)$ and $u+I\neq 0$.
In other words, $u\in \Soc(S/I)$ if and only if 
$u \not\in I$ and $ux_{i} \in I$ for all $1 \leq i \leq n$.

The present paper is organized as follows.  First, in Section $1$, we show that,
for a squarefree monomial ideal $I \subset S$, 
if a monomial $x_1^{a_1}\cdots x_n^{a_n}$ is a socle element of $S/I^k$, then
$a_i\leq k-1$ for all $1 \leq i \leq n$ (Corollary \ref{powersocle}).
Second, in Section $2$, the edge ideal $I_{G}$
arising from a finite graph $G$ is discussed.  We give a combinatorial 
characterization of $G$ for which $\depth S/(I_{G})^{2}=0$ (Theorem \ref{easy}).

Let $I \subset S$ be a squarefree monomial ideal. 
Then Corollary \ref{powersocle} says that one has $x_{[n]}^{k-1} \in \Soc(S/I^{k})$
if and only if $x_{[n]}^{k-1} \not\in I^{k}$, where
$x_{[n]} = x_{1}x_{2}\cdots x_{n}$.
In Section $3$, we study
squarefree monomial ideals $I \subset S$ with $x_{[n]}^{k-1} \in \Soc(S/I^{k})$.
It is proved that, for a squarefree monomial ideal $I \subset S$ with $x_{[n]}^{k-1}\in \Soc(S/I^k)$,
one has $k < n$ and $\depth S/I^j>0$ for $j< k$ (Corollary \ref{smallern}).
Furthermore, 
for a squarefree monomial ideal $I \subset S$ generated in degree $d$
with $x_{[n]}^{k-1}\in \Soc(S/I^k)$,
we show that if $d> ((k-1)n+1)/k$, then $\depth S/I^k> 0$
and that if $d= ((k-1)n+1)/k$ and $\depth S/I^k=0$, then $x_{[n]}^{k-1}\in \Soc(S/I^k)$  
and $\depth S/I^\ell=0$ for all $\ell\geq k$ (Corollary \ref{however}).

\section{Socles of powers of squarefree monomial ideals}
\label{socle}

We begin with

\begin{Proposition}
\label{easy}
Let $I$ be a  monomial ideal. For $i=1,\ldots,n$  set $$c_i=\max\{\deg_{x_i}(u) : u\in G(I)\},$$ and
let $x_1^{a_1}\cdots x_n^{a_n}$ be a socle element of $S/I$. Then $a_i\leq c_i-1$ for $i=1,\ldots,n$.
\end{Proposition}

\begin{proof}
Let $u = x_1^{a_1}\cdots x_n^{a_n}$ be a socle element of $S/I$.
Thus $u \not\in I$ and $u \in I : \mm$.
Suppose that  $a_{i} \geq c_i$ for some $i$.   Since $x_{i} u \in I$, there exists
 $v\in G(I)$ which divides $x_{i} u$.

It follows that $\deg_{x_j}(v)\leq \deg_{x_j}(x_iu)=\deg_{x_j}(u)$ for $j\neq i$, and $\deg_{x_i}(v)\leq c_i\leq \deg_{x_i}(u)$. Therefore, $v$ divides $u$, and hence $u\in I$, a contradiction.
\end{proof}

\begin{Corollary}
\label{powersocle}
Let $I$ be a squarefree monomial ideal, and let $x_1^{a_1}\cdots x_n^{a_n}$ be a socle element of $S/I^k$. Then
\[
a_i\leq k-1 \quad\text{for} \quad  i=1,\ldots,n.
\]
\end{Corollary}

\section{Edge ideals  whose square has depth zero}
\label{edge}

We consider the case of edge ideals.

\begin{Theorem}
\label{easy}
Let $I = I_{G}\subset S=K[x_1,\ldots,x_n]$ be the edge ideal of graph $G$ on the vertex set $[n]$.
The following conditions are equivalent:
\begin{enumerate}
\item[{\rm (a)}] $\depth S/I^2=0$;
\item[{\rm (b)}] $G$ is a connected graph  containing  a cycle $C$ of length $3$,
and any vertex of $G$ is a neighbor of $C$.
\end{enumerate}
Moreover, $x_{[n]}\in \Soc(S/I^2)$ if and only if $G$ is cycle of length $3$.
\end{Theorem}
\begin{proof}
(b) $\Rightarrow$ (a):
Suppose that $G$ has a cycle of length $3$, say,
$\{1, 2\}$, $\{1, 3\}$ and $\{2, 3\}$ are edges of $G$
and that, for each $4 \leq j \leq n$, one of $\{1, j\}$, $\{2, j\}$
and $\{3, j\}$ is an edge of $G$.  It then follows immediately that
the monomial $u = x_{1}x_{2}x_{3}$ satisfies
$u \not\in I^{2}$
and $u \in I^{2} : \mm$.
Hence $\depth S/I^{2} = 0$, as required. This argument also shows that $x_{[n]}\in \Soc(S/I^2)$ if and only if $G$ is cycle of length $3$

(a) $\Rightarrow$ (b):
Let $I = I_{G}$ be the edge ideal of a finite graph $G$
with $\depth S/I^{2}=0$.
Then there exists a monomial $u$ with $u \not\in I^{2}$
such that $u \in I^{2} : \mm$.
Let $H$ denote the induced subgraph of $G$
whose vertices are those $i \in [n]$ such that $x_{i}$ divides $u$.
Since $u \not\in I^{2}$ it follows that
$H$ cannot possess two disjoint edges.
If $H$ possesses an isolated vertex $i$, then $x_{i} u \not\in I^{2}$.
This contradict $u \in I^{2} : \mm$.
Hence $H$ is connected without disjoint edges.
Thus $H$ must be either a cycle of length $3$, or a line of length at most $2$.

First, if $H$ is a line of length $1$, i.e., $H$ is an edge of $G$,
then we may assume that $u = x_{1}^{a_{1}}x_{2}^{a_{2}}$ with each $a_{i} \geq 1$.
If each $a_{i} \geq 2$, then $u \in I^{2}$, a contradiction.  Let $a_{1} = 1$
and $u = x_{1}x_{2}^{a_{2}}$.  Then $u x_{2} \not\in I^{2}$.
This contradict $u \in I^{2} : \mm$.

Now,
let $H$ be either a cycle of length $3$, or a line of length $2$.
Thus we may assume that $u = x_{1}^{a_{1}}x_{2}^{a_{2}}x_{3}^{a_{3}}$
with each $a_{i} \geq 1$, where $\{1, 2\}$ and $\{1, 3\}$ are edges of $G$.
Since $u \not\in I^{2}$, it follows that $a_{1} = 1$.  Thus
$u = x_{1}x_{2}^{a_{2}}x_{3}^{a_{3}}$.  If $\{2, 3\}$ is not an edge of $G$,
then $x_{2} u \not\in I^{2}$, a contradiction.  Hence
$\{2, 3\}$ is an edge of $G$.  Then, since $u \not\in I^{2}$,
it follows that $a_{2} = a_{3} = 1$.  Thus $u = x_{1}x_{2}x_{3}$
and $\{1, 2\}$, $\{1, 3\}$ and $\{2, 3\}$ are edges of $G$.
Let $j \geq 4$.  Since $x_{j} u \in I^{2}$,
it follows that one of $\{1, j\}$, $\{2, j\}$ and $\{3, j\}$
must be an edge of $G$, as desired.
\end{proof}

This result has been shown independently by \cite{TT}.

\section{Powers of squarefree monomial ideals with maximal socle}
\label{maximal}

 Let $I\subset S=K[x_1,\ldots,x_n]$ be a squarefree monomial ideal. According to Corollary~\ref{powersocle},  $x_{[n]}^{k-1}$ is a socle element of $S/I^k$ if  $x_{[n]}^{k-1}\not\in I^k$. In that case it is a socle element of $S/I^k$ of maximal degree.  The next proposition characterizes those squarefree monomial ideals for which $x_{[n]}^{k-1}$ is indeed a socle element of $S/I^k$.

 We consider $I$ as the facet ideal of a simplicial complex $\Delta$. Thus $I=I(\Delta)$ where the set of facets $\Fc(\Delta)$ of $\Delta$ is given as
\[
\Fc(\Delta)=\{\supp(u)\:\; u\in G(I)\}.
\]
In other words, $G(I(\Delta))=\{x_F\:\; F\in \Fc(\Delta)\}$ where we set  $x_F=\prod_{i\in F}x_i$ for $F\subset [n]$.

\begin{Proposition}
\label{all}
Let $\Delta$ be a simplicial complex on the vertex set $[n]$, and $I=I(\Delta)\subset S=K[x_1,\ldots,x_n]$ its facet ideal.
\begin{enumerate}
\item[{\rm (a)}] The following conditions are equivalent:
\begin{enumerate}
\item[{\rm (i)}]  $x_{[n]}^{k-1}\not\in I^k$;
\item[{\rm (ii)}]  $\Sect_{i=1}^{k}F_i\neq \emptyset$ for all $F_1,\ldots,F_k\in \Fc(\Delta)$.
\end{enumerate}
\item[{\rm (b)}] Assuming that $x_{[n]}^{k-1}\not\in I^k$, the following conditions are equivalent:
\begin{enumerate}
\item[{\rm (i)}]  $x_jx_{[n]}^{k-1}\in I^k$ for all $j$;
\item[{\rm (ii)}] for each $j=1,\ldots,n$, there exist $F_1,\ldots,F_k\in \Fc(\Delta)$ such that $\Sect_{i=1}^{k}F_i=\{j\}$.
\end{enumerate}
\end{enumerate}
In particular,  $x_{[n]}^{k-1}\in \Soc(S/I^k)$ if and only if {\em (a)(ii)} and {\em (b)(ii)} hold.
\end{Proposition}

\begin{proof}
(a) $x_{[n]}^{k-1}\in I^k$ if and only if there exist $F_1,\ldots, F_k\in \Fc(\Delta)$ such that $x_{F_1}x_{F_2}\cdots x_{F_k}$ divides  $x_{[n]}^{k-1}$. This is the case, if and only if no $x_i^k$ divides $x_{F_1}x_{F_2}\cdots x_{F_k}$. This  is equivalent to saying  that $\Sect_{i=1}^{k}F_i=\emptyset$. Thus the desired conclusion follows.

(b) $x_ix_{[n]}^{k-1}\in I^k$ if and only if   $x_{F_1}x_{F_2}\cdots x_{F_k}$ divides  divides  $x_jx_{[n]}^{k-1}$ for some $F_1,\ldots,F_k\in \Fc(\Delta)$. By (a), $\Sect_{i=1}^{k}F_i\neq \emptyset$. Therefore, $x_{F_1}x_{F_2}\cdots x_{F_k}$  divides  $x_jx_{[n]}^{k-1}$  if and only if $\Sect_{i=1}^{k}F_i=\{j\}$.
\end{proof}

\begin{Corollary}
\label{smallern}
Let $I\subset S=K[x_1,\ldots,x_n]$ be a squarefree monomial ideal.  Let $n>1$ and suppose that  $x_{[n]}^{k-1}\in \Soc(S/I^k)$.  Then $k < n$, and $\depth S/I^j>0$ for  $j< k$.
\end{Corollary}

\begin{proof}
The condition (b)(ii) of Proposition~\ref{all} guarantees the existence
of $F^{(j)} \in \Fc(\Delta)$ with $j \in F^{(j)}$ and $j+1 \not\in F^{(j)}$
for each $1 \leq j < n$ and the existence of
$F^{(n)} \in \Fc(\Delta)$ with $n \in F^{(n)}$ and $1 \not\in F^{(n)}$.
Then $\Sect_{j=1}^{n} F^{(j)} = \emptyset$.  Thus if $k \geq n$, then
the condition (a)(ii) of Proposition~\ref{all} is violated, and hence $k<n$.

Let $j<k$ and suppose that $\depth S/I^j=0$. Then $j\geq 2$, since $I$ is squarefree. Let $u\in \Soc(S/I^j)$; then $ux_i\in I^j$ for all $i$ and hence also $x_{[n]}^{j-1}x_i\in I^j$ for all $i$. Since $n>1$, the ideal $I$ cannot be a principal ideal, 
because otherwise $\depth S/I^j>0$ for all $j$. Hence there exist an integer $i$, say $i=1$, such that $x_2x_3\cdots x_n\in I$. Then
\[
x_{[n]}^j=(x_{[n]}^{j-1}x_1)(x_2x_3\cdots x_n)\in I^{j+1}.
\]
It follows that
\[
x_{[n]}^{k-1}=x_{[n]}^{j}x_{[n]}^{k-j-1}=(x_{[n]}^jx_1^{k-j-1})(x_2x_3\cdots x_n)^{k-j-1}\in I^k,
\]
a contradiction.
\end{proof}

\begin{Examples}
\label{fullsocle}
{\em
(a) The ideal $$I=(x_1x_2\cdots x_{n-1}, x_1x_n,x_2x_n, \ldots, x_{n-1}x_n)$$ in $S=K[x_1,\ldots,x_n]$ satisfies the conditions  (a)(ii) and  (b)(ii) of Proposition~\ref{all} for $k=2$. Hence $\depth(S/I^2)=0$.

(b) Let $n = 2d - 1$ and $I$ a monomial ideal of $S = K[x_1,\ldots,x_n]$
generated by squarefree monomials of degree $d$.
Then the condition (a)(ii)  in Proposition~\ref{all} is satisfied for $k=2$.
Thus if a squarefree monomial $w$ belongs to $\Soc(S/I^2)$, then
$w$ must be $x_{[n]}$.  Hence $\depth S/I^2=0$ if and only if $I$ satisfies for $k=2$
the condition (b)(ii)  in Proposition~\ref{all}.

For example, if $I$ is generated by the following  squarefree monomials
\begin{eqnarray*}
& & x_{1}x_{2} \cdots x_{d}, \, \, \, \, \, x_{1}x_{d+1} x_{d+2}
\cdots x_{2d-1},\\
& & x_{i} x_{d+1}x_{d+2} \cdots x_{2d-1} \quad \text{with} \quad 2
\leq i \leq d,\\
& & x_{2} x_{3} \cdots x_{d} x_{j} \quad \text{with} \quad d+1 \leq j \leq 2d-1,
\end{eqnarray*}
then $\depth S/I^2=0$.}
\end{Examples}

Example~\ref{fullsocle}(b) shows that for any odd integer $n>1$ there exists a squarefree monomial ideal $I\subset K[x_1,\ldots,x_n]$ generated in degree $d=(n+1)/2$ such that $\depth S/I^2=0$.

\medskip
On the other hand for a squarefree monomial ideal generated in degree $d>(n+1)/2$ one has $\depth S/I^2>0$, as follow from

\begin{Corollary}
\label{however}
Let $I\subset K[x_1,\ldots,x_n]$ be a squarefree monomial ideal generated in the single  degree $d$.
\begin{enumerate}
\item[{\rm (a)}] If  $d> ((k-1)n+1)/k$,  then $\depth S/I^k> 0$.

\item[{\rm (b)}] For all positive integer $d,k$ and $n$ such that $d=((k-1)n+1)/k$, there exists a squarefree monomial ideal  $I\subset K[x_1,\ldots,x_n]$ generated in degree $d$ such that $\depth S/I^k=0$.
    
\item[{\rm (c)}] If $d= ((k-1)n+1)/k$ and $\depth S/I^k=0$. Then $x_{[n]}^{k-1}\in \Soc(S/I^k)$   and $\depth S/I^\ell=0$ for all $\ell\geq k$. 
 \end{enumerate}
\end{Corollary}

\begin{proof} (a) Let $F_1,\ldots,F_k$ subset of $[n]$ of cardinality $d$. We first show by induction on $i$, that $|\Sect_{j=1}^iF_j|> ((k-i)n+i)/k$. The assertion is trivial for $i=1$. By using the induction hypothesis, we see that
\[
|\Sect_{j=1}^iF_j|\geq |\Sect_{j=1}^{i-1}F_j|+|F_i|-n> \frac{(k-i+1)n+(i-1)}{k}+\frac{(k-1)n+1)}{k}-n=\frac{(k-i)n+i}{k},
\]
as desired.

It follows that  any intersection of $k$ subsets of $[n]$ of cardinality $d$ admits more than one elements. Therefore  $I$ satisfies condition (a)(ii) of Proposition~\ref{all}, but violates condition (b)(ii).

Since condition (a)(ii) is satisfied, it follows from Proposition~\ref{all} that $x_{[n]}^{k-1}\not\in I^k$. Thus, if we assume that $\depth S/I^k=0$, Corollary~\ref{powersocle} implies that $x_{[n]}^{k-1}\in \Soc(S/I^k)$. However, since condition (b)(ii) is violated, this is not possible.

(b) Suppose that $d=((k-1)n+1)/k$. Then there exists an integer $r\geq 0$ such that $d=(r+1)k-r$ and $n=(r+1)k+1$. Consider the monomial ideal  $I$ generated by all squarefree monomials of degree $d$ in $K[x_1,\ldots,x_n]$. By   \cite[Corollary 3.4]{HH} one has
$$\depth S/I^k=\max\{0,n-k(n-d)-1\}.$$
Since $n-k(n-d)-1=(r+1)k+1-k(r-1)-1=0$, the assertion follows.

(c) Let $u\in \Soc(S/I^k)$,  $u=x_1^{a_1}x_2^{a_2}\cdots x_n^{a_n}$. Then, by Corollary~\ref{powersocle},  $a_i\leq k - 1$ for all $i$, and hence $\deg u\leq (k-1)n=kd-1$. On the other hand, since $ux_i\in I^k$, it follows that $\deg u+1\geq kd$. Thus we conclude that $\deg u=kd-1=(k-1)n$, which is only possible if $u=x_{[n]}^{k-1}$. Let $\ell >k$ and let $v$ be a generator of $I^{\ell-k}$.  Then $uvx_i\in I^{\ell+1}$, but $uv\not\in I^{\ell}$, 
because 
\[
\deg uv=(kd-1)+(\ell - k) \leq kd - 1 + (\ell - k)d = \ell d - 1 <
\ell d.
\] 
This shows that $uv\in \Soc(S/I^\ell)$, and consequently $\depth S/I^{\ell}=0$, as required.
\end{proof}

\begin{Examples}
\label{allk}
{\em 
Let $k\geq 2$, and assume that $d=((k-1)n+1)/k$. Then $n=(kd-1)/(k-1)$, and this is an integer if and only if $d\equiv 1\mod(k-1)$. One solution is $d=k$. Then $n=k+1$. With these data we may choose the ideal $I\subset S=K[x_1,\ldots,x_n]$ generated  by all squarefree monomials of degree $d=k=n-1$. Then obviously $I$ satisfies the conditions (a)(i) and (b)(i) of Proposition~\ref{all}.  Thus $x_{[n]}^{k-1}\in \Soc(S/I^k)$. In particular, $\depth S/I^k=0$. It is shown in \cite{HH} that $\depth S/I^j>0$ for $j<k$. (This also follows from Corollary~\ref{smallern}). This example shows that arbitrary high powers of a squarefree monomial ideal may have a maximal socle.
}
\end{Examples}

It is known by a result of Brodmann \cite{B} (see also \cite{HH})  that the depth function $f(k)=\depth S/I^k$ is eventually constant. In \cite{HRV} the smallest number $k$ for which $\depth S/I^k= \depth S/I^j$ for all $j\geq k$, is denoted by $\dstab(I)$. In \cite{HQ} it is conjectured that $\dstab(I)<n$ for all graded ideals in $K[x_1,\ldots,x_n]$. Corollary~\ref{smallern} together with Corollary~\ref{however}(c) show that this conjecture holds true for a squarefree monomial ideal $I\subset K[x_1,\ldots,x_n]$ generated in degree $d=((k-1)n+1)/k$ for which  $\depth S/I^k=0$.


\end{document}